\documentclass[11 pt]{amsart}

\usepackage{amssymb,color,enumerate}
\usepackage{latexsym}
\usepackage{graphicx}
\usepackage{pdfsync}
\usepackage{amsmath}
\usepackage{amsfonts, ulem}
\usepackage{natbib}

\allowdisplaybreaks

\def\HDA{hidden domain of attraction }
\def\MRV{multivariate regular variation }
\def\bX{\boldsymbol X}

\def\bone{\boldsymbol 1}

\def\bzero{\boldsymbol 0}

\def\bx{\boldsymbol x}
\def\bz{\boldsymbol z}

\def\by{\boldsymbol y}
\def\binfty{\boldsymbol \infty}

\iftrue 
\usepackage{amsmath}
\usepackage{amsfonts}
\newcommand{\field}[1]{\mathbb{#1}}
\DeclareMathOperator{\PR}{\field{P}}             
\DeclareMathOperator{\E}{\field{E}}              
\def\N{\field{N}}                                
\def\R{\field{R}}                                
\def\F{\field{F}}                                

\else
\def\PR{\mathop{\rm I\kern -0.20em P}\nolimits}  
\def\E{\mathop{\rm I\kern -0.20em E}\nolimits}   
\def\N{\mathop{\rm I\kern -0.20em N}\nolimits}   
\def\R{\mathop{\rm I\kern -0.20em R}\nolimits}   
\def\F{\mathop{\rm I\kern -0.20em F}\nolimits}   
\fi

\newtheorem{thm}{Theorem}[section]

\newtheorem{prop}[thm]{Proposition}
\newtheorem{ex}[thm]{Example}
\newtheorem{defn}[thm]{Definition}
\newtheorem{rem}[thm]{Remark}

\numberwithin{equation}{section}

\title[Hidden domain of attraction]{Modeling Multiple Risks: Hidden domain of attraction} 

\author[ A.\ Mitra ]{Abhimanyu Mitra }
\address{Abhimanyu Mitra\\School of OR\&IE, Cornell University,
Ithaca, NY-14853} \email{am492@cornell.edu}

\author[ S. I.\ Resnick ]{Sidney I. Resnick}
\address{Sidney I. Resnick\\School of OR\&IE, Cornell University,
Ithaca, NY-14853} \email{sir1@cornell.edu}

\keywords{Regular variation, maximal domain of attraction,  spectral measure, risk sets.} 
\thanks{S. Resnick and A. Mitra were partially supported by ARO Contract W911NF-10-1-0289 at Cornell University.}

\normalsize

\begin{document}

\begin{center}
\maketitle

\end{center}
\normalsize
\vspace{-.2in}
\begin{abstract}
Hidden regular variation is a sub-model of multivariate regular
variation and facilitates accurate estimation of joint tail
probabilities. We generalize the model of hidden regular
variation to what we call hidden domain of attraction. We exhibit
examples that illustrate the need for a more general model and  discuss
detection and estimation techniques. 
\end{abstract}

\bibliographystyle{abbrv}
\section{Introduction} \label{intro}
Tail probabilities, especially joint tail probabilities,  provide useful
{risk measures for} many  applications including finance
\citep{poon:rockinger:tawn:2003}, environmental protection
\citep{smith:2003} and hydrology \citep{dehaan:deronde:1998,
  bruun:tawn:1998}. Multivariate extreme value theory (MEVT) is
a  tool to approximate such tail probabilities but in many common
circumstances the tool gives an incorrect tail probability approximation of $0$.
This paper points out that even when hidden regular variation (HRV) is not
applicable, a more general concept called {\it hidden domain of
  attraction\/} may yield a fix.

The joint distribution $H(\cdot)$ of a bivariate random vector ${\bX }
= (X^1, X^2)$ belongs to {the} maximal domain of attraction of a bivariate distribution
$G(\cdot)$  if there exist scaling and centering constants
$a_n^i>0$ and $b_n^i$, $i=1,2$,  such that 
for all
continuity points ${\bx } = (x^1, x^2)$ of $G$,
 \begin{equation}\label{eqn:mevt_original}
 \lim_{n \to \infty} { \left[ H(a^1_n x^1 + b^1_n , a^2_n x^2 +
 b^2_n )\right]}^n = G(x^1, x^2)
\end{equation}
and both the marginal distributions of $G(\cdot)$, $G^1(\cdot)$ and
$G^2(\cdot)$, are non-degenerate extreme value distributions
\cite[page 208]{dehaan:ferreira:2006}. The convergence relation
\eqref{eqn:mevt_original} is equivalent to the condition that as $n
\to \infty$, 
\begin{equation}\label{eqn:mevt}
n P\left[ \left( \frac{X^1 - b^1_n}{a^1_n}, \frac{X^2 -
 b^2_n}{a^2_n} \right) \in \cdot \,\right] \stackrel{v}{\rightarrow} \nu(\cdot)
\end{equation}
in $M_+(\E)$, where $\E = [-\infty, \infty]^2	 \setminus \{
(-\infty, -\infty)\}$ or $\E = [0, \infty]^2 \setminus \{ (0, 0)\}$ or
$\E = [-\infty, \infty] \times [0, \infty] \setminus \{ (-\infty, 0)
\}$ or $\E = [0, \infty] \times [-\infty, \infty] \setminus \{ (0,
-\infty) \}$ depending on the {case.}
{Also}, $M_+(\E)$ denotes the set of Radon measures on $\E$ and
$\stackrel{v}{\rightarrow}$ denotes vague convergence. The limit
measure $\nu(\cdot)$ in \eqref{eqn:mevt} is related to the limit
distribution $G(\cdot)$ in \eqref{eqn:mevt_original} as follows: for
${\bx } = (x^1, x^2) \in \E$, 
\begin{equation}\label{eqn:nuandG}
\nu( {\{(z^1, z^2) \in \E: z^1 \le x^1, z^2 \le x^2\} }^c ) = -\log(G(x^1, x^2)).
\end{equation} 
Assuming  $[X^1 > u, X^2 > v ]$ is a rare event, that is, that $u$ and
$v$ are sufficiently large, we use MEVT to
approximate the joint tail probability $P( X^1 > u, X^2 
> v)$  as
\begin{equation}\label{eqn:add1}
 P( X^1 > u, X^2 > v) \approx \frac1{n} \nu\left( \left(\frac{u- b^1_n}{a^1_n}, \infty \right] \times \left(\frac{v- b^2_n}{a^2_n}, \infty \right] \right).
 \end{equation}
However, in the presence of asymptotic independence \cite[page
226]{dehaan:ferreira:2006}, \eqref{eqn:add1}  approximates the joint tail probability
$P( X^1 > u, X^2 > v)$ as zero. Perhaps this approximation is
crude and a better estimate is possible.

If  \eqref{eqn:mevt} holds with $\E = [0, \infty]^2 \setminus \{
(0, 0)\}$, $b^1_n = b^2_n = 0$ and some $a^1_n, a^2_n \uparrow
\infty$, we obtain multivariate regular variation
(MRV). {I}f $X^1$ and $X^2$ are {also} asymptotically
independent, we {may improve the approximation of}
 joint tail probabilities {if} 
hidden regular variation (HRV)  is present; see \citet{resnick:2002a, 
mitra:resnick:2010a,
resnickbook:2007}
and the seminal 
\citet{ledford:tawn:1996, ledford:tawn:1998}. However, HRV requires
the distribution of {$X^1 \wedge X^2$} to have a regularly varying
tail and this may not be the case.  Perhaps
$X^1 \wedge X^2$ has a distribution in some maximal domain of
attraction 
other than the heavy tailed domain. In this case,
HRV cannot be applied to improve
  joint tail probability approximation 
but the deficiency can be remedied by a more general approach
which we call {\it hidden
domain of attraction\/} (HDA). HRV is a special case of HDA. 

If the distribution of ${\bX }$ 
does not have MRV but \eqref{eqn:mevt} still holds,
we may retrieve the MRV setup by transforming the
 components of ${\bX }$ to $\left(U^1(X^1), U^2(X^2)\right)$, where
 $U^i(\cdot) = 1/(1-H^i(\cdot))$ and $H^i(\cdot)$ is the distribution
 of $X^i$, $i =1, 2$ \cite[page 265]{resnickbook:2008}. If  $X^1$ and
 $X^2$ are asymptotically independent, {so are $\left(U^1(X^1),
   U^2(X^2)\right)$ and assuming $U^1(X^1) \wedge U^2(X^2)$ has a
 regularly varying tail, we may seek HRV.} Statistically this is
problematic since we do not know
 $U^i(\cdot)$, $i =1, 2$. This can be dealt with in various ways, none
 of which is completely satisfying or easy and a potential
{advantage of the notion of HDA is
 that it does not require that we transform components.}

 \subsection{Outline} Section \ref{sec:notation} reviews frequently
 used notation. In Section \ref{sec:standard}, we define \HDA for the
 standard case, when both the components of the risk vector have the
 same distribution. Section \ref{sec:nonstandard} deals with HDA in
 the non-standard case, where we 
{drop the identical distribution assumption for $X^1,X^2$.}
In both 
Sections \ref{sec:standard} and \ref{sec:nonstandard},
we  exhibit examples which satisfy our model and discuss estimation 
 procedures of limit measures that appear in the limit relations of
 the model. Section \ref{sec:detectionestimation} discusses the
 detection techniques for HDA and estimation of joint and marginal
 tail probabilities. We conclude with a few remarks in Section
 \ref{sec:conclusion}.

 \subsection{Notation}\label{sec:notation}
{{For simplicity}}, this paper is restricted to two dimensions.
  For  denoting a vector and its components, we use:
$${\bx } = (x^1, x^2), \hskip 1 cm x^i = \hbox{$i$-th component of } {\bx }, \hskip 0.1 cm i = 1, 2.$$
{Multivariate intervals or rectangles are denoted $(\bx,\by],
  [\bx,\by]$, etc where, for instance, $(\bx,\by]= (x^1,y^1]\times (x^2,y^2].$}
The vectors of all zeros, all ones and all infinities are denoted by
$\bzero= (0, 0),$ $\bone = (1, 1)$ and
${\boldsymbol{\infty}} = (\infty, \infty)$
respectively. 
{We write}
$ x^{(1)} = x^1\vee x^2,\qquad  x^{(2)}=x^1\wedge x^2.$
So, the superscripts denote components of a vector and the ordered
component is denoted by a parenthesis in the superscript.

 We express  vague convergence \citep[page 173]{resnickbook:2007} of Radon measures as
$\stackrel{v}{\rightarrow}$ and
weak convergence of probability measures \cite[page
14]{billingsley:1999}
as $\Rightarrow$. Denote a point measure with points  $\{x_i\}$ in a
nice space  $\F$  by 
$\sum_i \epsilon_{x_i}$ where for $x\in \F$ and $B\subset \F$,
$$\epsilon_{x} (B)=\begin{cases}
1,& \text{ if }x \in B,\\
0,& \text{ if }x \in B^c.
\end{cases} $$
Write $M_+(\F)$ for  the set of non-negative Radon
measures on a space $\F$ topologized by the vague topology. 

{For a one dimensional
distribution $F(x)$,  set $\bar F:=1-F$.} The inverse of
a non-decreasing function $\psi (x)$ is $\psi^\leftarrow (x)$.
 
\section{Standard case \HDA}\label{sec:standard}
Suppose that a bivariate random vector ${\bX } = (X^1, X^2) $ with
distribution $H(\bx)$
belongs to the maximal domain of attraction of an extreme value
distribution $G$ \citep[page 265]{resnickbook:2008}, $X^1
\stackrel{d}{=} X^2$  and $X^1$ and $X^2$ are asymptotically
independent so that \eqref{eqn:mevt_original} is satisfied  with
$b^1_n = b^2_n = b_n$, $a^1_n = a^2_n = a_n{>0}$ and $G(x, y) =
G^1(x)G^2(y)$ for $x, y \in \R$, $n \in \mathbb{N}$. Thus,
\eqref{eqn:mevt} becomes  
\begin{align}\label{doabigcone}
nP\left[ \left( \frac{\bX - b_n\bone}{a_n} \right) \in \cdot \, \right] \stackrel{v}{\rightarrow} \nu(\cdot)
\end{align}
in $M_+(\E)$, where $\E = [-\binfty, \binfty] \backslash \{ -\binfty
\}$ or $\E = {[\bzero, \binfty]} \backslash \{ \bzero \}$. Since $G(x,
y) = G^1(x)G^2(y)$ for $x, y \in \R$, the relation of $\nu$ and
$G$ given in \eqref{eqn:nuandG} gives for ${\bx } \in \E$, 
\begin{align}\label{eqn:nustructure}
\nu( {\{{\bf{z}} \in \E: z^1 \le x^1, z^2 \le x^2\} }^c ) =&
-\log G^1(x^1) +  -\log G^2(x^2)\nonumber \\
=& \nu( {\{{\bf{z}} \in \E: z^1 \le x^1\} }^c ) + \nu( {\{{\bf{z}} \in \E: z^2 \le x^2\} }^c ).
\end{align} 
{T}he {standard case contains the}  additional assumption that
$X^1 \stackrel{d}{=} X^2$, which reduces \eqref{eqn:mevt} to
\eqref{doabigcone} and reduces  possible choices for $\E$. The cone
$\E = {[\bzero, \binfty]} \backslash \{ \bzero\}$ is chosen only when $H$
has MRV.

{}From \eqref{doabigcone}, the maximal component of $\bX$ satisfies as
$n\to\infty$, 
\begin{equation}\label{eqn:maxdoa}
nP( X^{(1)} > a_n y + b_n) \to \nu( \{{\bf{z}} \in \E: z^{(1)} > y
\}),\qquad (y, y) \in \E, 
\end{equation}
so $X^{(1)}$ is in a maximal domain of attraction of an extreme value
distribution and the distribution of $X^{(1)}$ characterizes $a_n$ and
$b_n$ given in \eqref{doabigcone}. Using one-dimensional extreme value
theory, we can choose $a_n$ and $b_n$ in such a way that  
\begin{equation}\label{defn:psi}
\psi(y) := \nu( \{{\bf{z}} \in \E: z^{(1)} > y \})
\end{equation}
takes one of the following forms: 
\begin{align*}
\begin{array}{lc}
\psi(y) = \left \{ \begin{array}{cc}
y^{1/\gamma}, & \hbox{if $y > 0$,}\\
0, & \hbox{otherwise,}
\end{array} \right. & \hbox{if $ \gamma > 0,$}\\
 \psi(y) = e^{y}, & \hbox{if $ \gamma = 0,$}\\
 \psi(y) = \left \{ \begin{array}{cc}
\infty, & \hbox{if $y > 0$,}\\
(-y)^{-1/\gamma}, & \hbox{otherwise,}
\end{array} \right. & \hbox{if $ \gamma < 0,$}
\end{array}
\end{align*}
where $\gamma$ is the extreme value index of the distribution of $X^{(1)}$ \cite[page 9]{resnickbook:2008}. This remains our standing assumption for the following discussion.

We define a sub-model of MEVT {called (standard case) hidden domain of
attraction (HDA)}. HDA  helps approximate  joint tail 
probabilities in the presence of asymptotic independence and 
 includes HRV as a
special case. If $\E$ is either $[-\binfty, \binfty] \backslash \{ -\binfty \}$
or $ {[\bzero, \binfty]} \backslash \{ \bzero \}$, define
$\E^0$ as ${(-\binfty, \binfty]}$ or $ (\bzero, \binfty]$ but
see Remark \ref{rem:hdoa}(\ref{rem3:hdoa}) before jumping to erroneous conclusions
that $\E^0\subset \E$ is always true.

\begin{defn}\label{defn:hdoa}
{\rm{The distribution of ${\bX } = (X^1, X^2 )$ has standard case hidden domain of
    attraction  on the cone $\E^0$ if  (i) $X^1 \stackrel{d}{=} X^2$;
(ii)  \eqref{doabigcone} and \eqref{eqn:nustructure} hold;  (iii)  there exist
{positive scaling  and real centering}  constants $\{ c_n  \}$ and $\{d_n \}$ and a
    non-zero measure $\nu^0 \in M_+(\E^0)$ such that in $M_+(\E^0)$,
\begin{align}\label{hiddendoasmallcone}
nP\left[  {(\bX -d_n \bone} )/{c_n}  \in \cdot \, \right] \stackrel{v}{\rightarrow} \nu^0(\cdot)
\qquad (n \to \infty).
\end{align}
}}
\end{defn}

We emphasize that the definition requires that $X^1
\stackrel{d}{=}X^2$ and  that the distribution
  of ${\bX }$ belongs to the maximal domain of 
  attraction of an extreme value product measure  $G$ {with exponent
  measure $\nu$}. Some other remarks:
\begin{rem}\label{rem:hdoa}
{\rm{
\begin{enumerate}
\item  Hidden regular variation assumes
  \eqref{doabigcone} is satisfied 
{on the cone $\E = {[\bzero, \binfty]} \backslash \{ \bzero\}$}
with $b_n = 0$ and $a_n \uparrow
  \infty$ 
  and \eqref{hiddendoasmallcone} is satisfied 
{on the cone $\E^0 = (\bzero, \binfty]$}
with $d_n = 0$ and $c_n
  \uparrow \infty$.  Moreover,
  $a_n/c_n \to \infty $ as $n \to \infty$.  Hidden regular
  variation  is a special case
  of hidden domain of attraction. HRV is the only sub-model of HDA where  the cone $\E^0$ in
  \eqref{hiddendoasmallcone} is $\E^0 = (\bzero, \binfty]$.  

\item \label{rem3:hdoa} From \eqref{hiddendoasmallcone} the minimum
  component of $\bX$ satisfies,
\begin{equation}\label{eqn:mindoa}
nP[ X^{(2)} > c_n y + d_n ] \rightarrow \nu^0((y, \infty] \times (y,
\infty]) \qquad ((y, y) \in \E^0),
\end{equation}
and therefore, the distribution of $X^{(2)}$ belongs to the maximal domain
of attraction of an extreme value distribution \citep[page
4]{dehaan:ferreira:2006}. When HRV exists, the distribution of
$X^{(2)}$ has a regularly varying tail and is hence in 
the domain of attraction of the Fr\'{e}chet distribution. HDA 
{allows} the additional cases where the distribution of
$X^{(2)}$ belongs to the domain of attraction of the Gumbel or the
 Weibull distribution. 

The distribution of $X^{(2)}$ determines the scaling and centering
constants $\{c_n\}$ and $\{d_n\}$ and the cone $\E^0$. As illustrated
by  Example \ref{eg:unif}, even
if $\E = [\bzero, \binfty] \setminus \{ \bzero \}$, the cone $\E^0$ could be
$(-\binfty, \binfty]$ {and} $\E^0$ is not necessarily a sub-cone of
$\E$, as was the case {for} HRV \citep{resnick:2002a}.

\item From \eqref{eqn:maxdoa}, we get that $X^{(1)}$ belongs to the
  maximal domain of attraction of some extreme value
  distribution. Since $X^{(1)} \ge X^{(2)}$, the convergence relation
  \eqref{doabigcone} puts some restriction as to what possible
  convergences can hold in \eqref{hiddendoasmallcone}. For example, if
  \eqref{doabigcone} is satisfied with $X^{(1)}$ being in the Gumbel
  domain of attraction, then HRV can never hold on $\E^0$ since the
  tail of $X^{(2)}$ cannot be heavier than the tail of $X^{(1)}$. 
\end{enumerate}
}}
\end{rem}

\subsection{Semi-parametric structure of
  $\nu^0$} \label{sec:semiparaoflimit} 
The limit measure $\nu^0$ in
\eqref{hiddendoasmallcone} has a semi-parametric structure which
assists estimation (as in \citep{mitra:resnick:2010a} for HRV) and
which characterizes the class of possible limit measures as a class
indexed by a real parameter and a set of probability measures.

 To understand this semi-parametric structure,  proceed as
 follows. Let  $H^{(2)}(\cdot)$ be the distribution of $X^{(2)}$ 
 {and} define the function $\psi^0(\cdot)$ as   
\begin{equation}\label{defn:psi0}
\psi^0(y) := {\left[\nu^0((y, \infty] \times (y, \infty])\right]}^{-1}.
\end{equation}
where $\nu^0(\cdot)$ is given in \eqref{hiddendoasmallcone}. From
\eqref{eqn:mindoa} we get  
\begin{align}\label{eqn: Hconv}
n \overline {H^{(2)}}(c_n y + d_n) \rightarrow
{[\psi^0(y)]}^{-1}\qquad (y \in \R).
\end{align}
 Hence, from the one-dimensional extreme value theory, 
$H^{(2)}$ is in a maximal domain of attraction.  Let $\gamma^0$ be the extreme value
index {of $H^{(2)}$}.  Assuming $c_n$ and $d_n$ are chosen suitably \citep[page
9]{resnickbook:2008}, $\psi^0(\cdot)$ must take  one of the following
three forms: 
\begin{align}\label{eqn:formsofpsi}
\begin{array}{lc}
 \psi^0(y) = \left \{ \begin{array}{cc}
y^{1/\gamma^0}, & \hbox{if $y >0$,}\\
0, & \hbox{otherwise,}
\end{array} \right. & \hbox{if $ \gamma^0 > 0,$}\\
 \psi^0(y) = e^{y}, & \hbox{if $ \gamma^0 = 0,$}\\
 \psi^0(y) = \left \{ \begin{array}{cc}
\infty, & \hbox{if $y > 0$,}\\
(-y)^{-1/\gamma^0}, & \hbox{otherwise,}
\end{array} \right. & \hbox{if $ \gamma^0 < 0.$}
\end{array}
\end{align}
Henceforth  assume that 
 $c_n$ and $d_n$ are chosen so that \eqref{eqn:formsofpsi} is true. Define  
\begin{equation}\label{defn:U}
U^{(2)}(x) = 1/(1 - H^{(2)}(x)),
\end{equation}
{and the} following {helps us
identify the semi-parametric structure of $\nu^0$. }

\begin{prop}\label{prop:equihrv}
{The convergence in \eqref{hiddendoasmallcone} that defines HDA is
  equivalent to the regular variation on $(\bzero, \binfty]$, 
\begin{equation}\label{equihrv}
nP\left[ n^{-1}\left( {U^{(2)}(X^1)}, {U^{(2)}(X^2)} \right)
  \in \cdot \,\right] \stackrel{v}{\rightarrow} \tilde \nu^0(\cdot) 
\end{equation}
where $U^{(2)}(\cdot)$ is defined in
\eqref{defn:U} and $\tilde \nu^0(\cdot)$ is a Radon measure on $(\bzero,
\binfty]$ that is related to the
limit measure  $\nu^0(\cdot)$  in \eqref{hiddendoasmallcone} by
\begin{equation}\label{nu0andtildenu0} 
\tilde \nu^0((x^1, \infty] \times (x^2, \infty] ) = 
\nu^0 \Bigl(
\bigl({(\psi^0 )}^{\leftarrow}(x^1), \infty \bigr]
  \times \bigl({(\psi^0 )}^{\leftarrow}(x^2), \infty
  \bigr] \Bigr), \qquad (\bx \in (\bzero, \binfty]).
\end{equation}
The measure $\tilde \nu^0(\cdot)$ satisfies the scaling property:
\begin{align}\label{scaling_nu0tilde}
\tilde \nu^0( c \cdot) = c^{-1} \tilde \nu^0( \cdot), \hskip 1 cm c > 0.
\end{align}
}
\end{prop}

\begin{rem}
{\rm{ 
\begin{enumerate}[(i)]
\item {Proposition \ref{prop:decomp} below shows} 
that the limit measure $\tilde \nu^0$ is
 determined by a probability measure $S^0$. Thus
the family of limits in \eqref{equihrv} is indexed by probability
measures and  Proposition \ref{prop:equihrv} shows that $\nu^0$ has semi-parametric structure:  the
  probability measure $S^0$ 
 determine{s} $\tilde \nu^0$ and given  $\gamma^0$,
  we  {get}  $\psi^0(\cdot)$ from
  \eqref{eqn:formsofpsi} and then 
applying \eqref{nu0andtildenu0}, we get
$ \nu^0$. 

\item {I}f the support of the distribution of $X^{(2)}$ is smaller
  than that of $X^i$, $i = 1, 2$, then $U^{(2)}(X^i)$ could take the
  value $\infty$ with positive probability. Hence, in the following
  discussion, we treat $U^{(2)}(X^i)$, $i = 1, 2$, as extended random
  variables.
\end{enumerate}
}}
\end{rem}

\begin{proof}[Proof of Proposition \ref{prop:equihrv}] {To see
    that}
 \eqref{equihrv} implies \eqref{hiddendoasmallcone}, observe that for ${\bx } \in \E^0$,
\begin{align*}
nP\left[ \frac{X^1 - d_n}{c_n} > x^1, \right. & \left. \frac{X^2- {d}_n}{{c}_n} >  x^2 \right] \\
&= nP\left[  \frac{U^{(2)}(X^1)}{n} > \frac{U^{(2)}(c_n x^1 + d_n)}{n}, \frac{U^{(2)}(X^2)}{n} > \frac{U^{(2)}(c_n x^2 + d_n)}{n}\right]\\
& \rightarrow \tilde \nu^0 \left( \left(\psi^0(x^1), \infty \right] \times \left(\psi^0(x^2), \infty \right] \right) = \nu^0( (x^1, \infty] \times (x^2, \infty]),
\end{align*}
where the convergence follows from \eqref{eqn: Hconv} and
\eqref{equihrv} and the last equality follows from
\eqref{nu0andtildenu0} and the forms of $\psi^0(\cdot)$ given in
\eqref{eqn:formsofpsi}. Hence, 
\eqref{hiddendoasmallcone} holds. 
The {converse} is similar and  is omitted.
\end{proof}

The scaling property \eqref{scaling_nu0tilde} implies that 
we can express \eqref{equihrv} in an alternate coordinate system that
transforms the limit measure into a product.
{}From Proposition \ref{prop:equihrv}, if the
distribution of ${\bX }$ satisfies Definition \ref{defn:hdoa} and  has HDA, then
$(U^{(2)}(X^1), U^{(2)}(X^2))$ has regular variation on $(\bzero,
\binfty]$. So, using \eqref{defn:psi0}, \eqref{eqn:formsofpsi} and
\eqref{nu0andtildenu0}, we get that 
\begin{align*}
\tilde \nu^0 \left((\bone, \binfty]  \right) =  \nu^0 \left({\left( {\left(\psi^0 \right)}^{\leftarrow}(1), \infty \right]}^2 \right) = {[\psi^0({\left(\psi^0 \right)}^{\leftarrow}(1)) ]}^{-1} = 1.
\end{align*}
This plus the scaling property \eqref{scaling_nu0tilde}
implies $\tilde \nu^0 ([\bone, \binfty]  ) =
1$. {This, Proposition 3.1 of  \cite{mitra:resnick:2010a} and
Proposition \ref{prop:equihrv} yield  the equivalent convergence in alternate coordinates
  given in \eqref{eqn:decomp}
  below.

For Proposition \ref{prop:decomp}, we need the following:
let $\nu_1$ be a Pareto
measure on $(0, \infty]$  satisfying $\nu_1((y, \infty]) = y^{-1}$ for
$y > 0$.  Since  $U^{(2)}(\cdot)$ is non-decreasing, $U^{(2)}(
    X^{(2)}) = U^{(2)}(X^1) \wedge U^{(2)}(X^2)$. Set
$$\delta \aleph^{(2)} =  \{ {\bx }  \in (0, \infty]^2 :
x^{(2)} = 1\}.$$  }

\begin{prop}\label{prop:decomp}
{{ The convergence in \eqref{equihrv} is equivalent to
\begin{equation}\label{eqn:decomp}
nP\left[ \left( \frac{U^{(2)}\left(X^{(2)}\right)}{n} , \left(
      \frac{U^{(2)}(X^1)}{U^{(2)}(X^{(2)})},
      \frac{U^{(2)}(X^2)}{U^{(2)}(X^{(2)})} \right) \right) \in \cdot \,
\right] \stackrel{v}{\rightarrow} \nu_1 \times  S^0(\cdot), 
\end{equation}
on $(0, \infty] \times \delta \aleph^{(2)}$, where  $S^0$ is a probability measure on $\delta
\aleph^{(2)}$. The relation between $\tilde \nu^0$  in
\eqref{equihrv} and $S^0$ is 
 \begin{equation}\label{tildenu0ands0}
 \tilde \nu^0 \left( \left \{ {\bx } \in (0, \infty]^2 : x^{(2)} \ge r, {\bx }/x^{(2)} \in \Lambda \right \} \right) = r^{-1} S^0(\Lambda),
 \end{equation}
for  $r >0$ and Borel sets $\Lambda \subset \delta \aleph^{(2)}$.
}}
\end{prop}

We call the probability measure $S^0$ {\it standardized
hidden spectral measure\/}.

HRV is a special case of HDA. For  HRV, a similar spectral
measure $S^{(2)}(\cdot)$ is defined \citep{mitra:resnick:2010a} 
on $\delta \aleph^{(2)}$
called {the} hidden spectral measure by
\begin{align*}
P\left[ {{\bX }}/{X^{(2)}} \in \cdot \,\Big{|} X^{(2)} > t \right]
\Rightarrow S^{(2)}(\cdot) \qquad {(t\to\infty).}
\end{align*}
 {The standardized hidden spectral
  measure $S^0(\cdot)$ on $\delta \aleph^{(2)}$ given in 
 \eqref{eqn:decomp} is}
\begin{align}\label{eqn:s0interpret}
P\left[  \frac{ \left(U^{(2)}(X^1), U^{(2)}(X^2)
    \right)}{U^{(2)}(X^{(2)})} \in \cdot \, \Big{|} X^{(2)} > t \right]
\Rightarrow S^0(\cdot),
\end{align}
where the convergence holds as $t \to
x_{H^{(2)}} = \sup \{ y \in \R: H^{(2)}(y) < 1 \} $.  If HRV exists,
$x_{H^{(2)}} = \infty$. 

\subsection{Examples}\label{subsec:egs}
{We give} examples of distributions that  possess
\MRV {with asymptotic independence. Each has \HDA
 but not hidden regular variation  emphasizing the need for
{a concept beyond HRV.}

\begin{ex}
{\rm{ Suppose, $W_1, W_2 \stackrel{iid}{\sim}  F(\cdot)$ and $Z_1, Z_2
    \stackrel{iid}{\sim} D(\cdot)$, where $F$ and $D$ belong to the
    maximal domain{s} of attraction of {the}  Fr\'{e}chet with
    index $\alpha =1$ and  Gumbel
    distribution{s} respectively.  Let $B$ be a Bernoulli random variable
    such that $P[ B = 1] = 0.5 = 1- P[ B = 0]$.   {A}ssume  the
    random variables $W_1, W_2, Z_1, Z_2$ and $B$ are mutually
    independent {and} define a bivariate random vector ${\bX }$ as 
$$ {\bX } = (X^1, X^2) = B(W_1, Z_1) + (1- B)(Z_2, W_2).$$
We show that the distribution of ${\bX }$ has MRV.
It {suffices} \citep{dehaan:1978}
to  verify that
$t_1X^1 \vee t_2X^2$ has a {regularly varying tail for any $t_1,
  t_2 > 0$.}  {Since $F$ has a regulary varying tail and}
\begin{equation*}
P[ t_1W_1 \vee t_2Z_1 > x] {\sim} P[t_1W_1 > x], \qquad (x \to
\infty),
\end{equation*}
$t_1X^1 \vee t_2X^2$ also has a regularly varying tail. Thus for
{appropriate $a_n \uparrow \infty$,}
\begin{align*}
nP\left[ {{\bX }}/{a_n} \in \cdot \, \right] \stackrel{v}{\rightarrow} \nu(\cdot)
\end{align*}
on $\E = [0, \infty]^2 \setminus \{{\bzero } \}$, where $\nu\left(
  {\left( [0, x^1] \times [0, x^2] \right)}^c \right) =
\frac{1}{2}\left({\left(x^1\right)}^{-1} +
  {\left(x^2\right)}^{-1}\right).$ Therefore, the distribution of
${\bX }$ has MRV with asymptotic independence on $\E$. 

Furthermore,  $X^{(2)}$ belongs to the maximal domain of
attraction of {a} Gumbel distribution and therefore,  HRV does not exist.
To see this,  without loss of generality
\citep{balkema:dehaan:1972, resnickbook:2008}, assume that $D$ is a
von-Mises function \citep[page 40]{resnickbook:2008} and for
specificity assume the right endpoint of $D$ is infinite. The form of
the tail is 
\begin{equation*}
\overline D (x) = c e^{- \int_1^x {\left[f_D(y) \right]}^{-1}dy} \quad
\text{and} \quad f_D'(x) \stackrel{x\to\infty}{\to} 0,
\end{equation*}
where $c>0$ is some constant. Likewise, assume without loss of generality
\citep[page 58]{resnickbook:2008} that $F$ satisfies $xF'(x)/\bar F(x)
\to 1$, where $1$ is the index of regular variation of $\bar F$.
Then  $\overline {H^{(2)}}(x) := \overline F(x) \overline D(x)$ is
{the tail of a} 
von-Mises function with auxiliary function  
\begin{equation*}
f_{H^{(2)}}(t) = tf_D(t)/(t + 1\cdot f_D(t)) \stackrel{t \to \infty}{\sim} f_D(t),
\end{equation*}
since
\begin{align*}
\frac1{f_{H^{(2)}}(t)} = \frac{{H^{(2)}}'(t)}{\overline {H^{(2)}}(t)} = \frac{F'(t)}{\overline F(t)} + \frac{D'(t)}{\overline D(t)} \stackrel{t \to \infty}{\sim}  t^{-1} + \frac1{f_D(t)}.
\end{align*}
From \cite[Corollary 1.7, page 46]{resnickbook:2008}  $H^{(2)}$
belongs to the maximal domain of attraction of {the} 
Gumbel
distribution.

Next we choose scaling and centering constants $\{ c_n \}$ and $\{ d_n
\}$ in \eqref{eqn: Hconv} so that $\psi^0(\cdot) = e^y$ (one of the
forms in \eqref{eqn:formsofpsi}). {The usual choices are \citep[page 40]{resnickbook:2008}}
$d_n ={ \left(1/{\overline {H^{(2)}}} \right)}^{\leftarrow}(n)$ and $c_n = f_{H^{(2)}}(d_n)$.

Now, observe that for ${\bx } \in \E^0 = (-\infty, \infty]^2$, as $n \to \infty$,
\begin{align*}
nP[ X^1 >& d_n+ c_nx^1,  X^2 > d_n + c_nx^2]  \\
&= \frac{n}{2}\overline F\left(d_n+ c_nx^1 \right)\overline D
\left(d_n+ c_nx^2 \right) + \frac{n}{2}\overline F\left(d_n+ c_nx^2
\right)\overline D \left(d_n+ c_nx^1 \right)\\ 
&= \frac{n}{2}\left( \frac{\overline F\left(d_n+ c_nx^1 \right)}{
    \overline F\left(d_n+ c_nx^2 \right)} \right)\overline {H^{(2)}}
\left(d_n+ c_nx^2 \right) 
+  \frac{n}{2}\left( \frac{\overline F\left(d_n+ c_nx^2
    \right)}{ \overline F\left(d_n+ c_nx^1 \right)} \right)\overline
{H^{(2)}} \left(d_n+ c_nx^1 \right) \\ 
& \rightarrow \frac{1}{2}( e^{-x^1} + e^{-x^2} ).
\end{align*}
 The convergence follows from the facts that $\overline F$ is regularly varying, $c_n/d_n \to 0$ and \eqref{eqn: Hconv} holds with $\psi^0(y) = e^y$. 
Therefore, as in Definition \ref{defn:hdoa}, the distribution of ${\bX }$  has HDA  on $\E^0 = (-\infty, \infty]^2$ with limit measure $\nu^0$ such that for ${\bx } = (x^1, x^2) \in \E^0$,
$$ \nu^0\left((x^1, \infty] \times (x^2, \infty] \right) = \frac{1}{2}( e^{-x^1} + e^{-x^2} ).$$

Thus the distribution of  ${\bX }$  is regularly varying on $\E$, has
HDA on $\E^0$, but does not have HRV since HRV
requires the distribution of $X^{(2)}$ to be in the domain of
attraction of the Fr\'{e}chet distribution \citep[page
54]{resnickbook:2008}. 
}}
\end{ex}

\begin{ex}\label{eg:unif}
{\rm{ 
Suppose, $U \sim$ Uniform($[0, 1]$). Define the random vector ${\bX }$ as 
$${\bX } = (X^1, X^2) = \left( \frac{1}{U}, \frac{1}{1-U} \right).$$
Now, note that for $x^1, x^2 > 0$, $2n > {(x^1)}^{-1} + {(x^2)}^{-1},$
\begin{align*}
n(1- P[ 1/U  & \le  2nx^1, 1/(1-U)  \le 2nx^2])
= n(1- P[ U \ge {(2nx^1)}^{-1}, U \le 1 - {(2nx^2)}^{-1}]) \\
&= n(1- (1 - {(2nx^2)}^{-1} - {(2nx^1)}^{-1}))
 \rightarrow \frac{1}{2}( {(x^1)}^{-1} + {(x^2)}^{-1} ),
\end{align*}
as $n \to \infty$. Therefore, on $\E = [0, \infty]^2 \setminus \{\bzero\}$,
$nP\left[ {{\bX }}/{2n} \in \cdot \right] \stackrel{v}{\rightarrow} \nu(\cdot)$
 where the limit measure $\nu$ satisfies
$\nu \left( {\left( [0, x^1] \times [0, x^2] \right)}^c \right) =
({(x^1)}^{-1} +
  {(x^2)}^{-1})/2,$ for ${\bx } \in \E$ and thus
 the distribution of ${\bX }$ has MRV with asymptotic independence.

Also note that for $\{ (x^1, x^2) \in (-\infty, \infty]^2: x^1 + x^2 \le 0\}$, and large $n$,
\begin{align}\label{ex2eqn:hdoa} 
nP\Bigl[ X^1 &> 2 + \frac{2x^1}{n+1}, X^2 >  2 + \frac{2x^2}{n+1}
\Bigr]  = nP\Bigl[ U < \frac{n+1}{2n + 2 + 2x^1}, U  >  \frac{n + 1 +
  2x^2}{2n+2 +2x^2} \Bigr] \nonumber \\ 
&= n \left(\frac{n+1}{2n + 2 + 2x^1} - \frac{n + 1 + 2x^2}{2n+2+ 2x^2} \right) 
= n \left( \frac{1}{2} - \frac{x^1}{2n + 2 + 2x^1} - \frac{1}{2} - \frac{x^2}{2n+2+2x^2} \right)\nonumber \\
& \rightarrow \frac{1}{2}( (-x^1) + (-x^2) ),
\end{align}
as $n \to \infty$. Similar calculations show that for $\{ (x^1, x^2) \in (-\infty, \infty]^2: x^1 + x^2 > 0\}$, 
\begin{align}\label{ex2eqn:hdoa2}
nP\left[ X^1 > 2 + \frac{2x^1}{n+1}, X^2 >  2 + \frac{2x^2}{n+1} \right] \rightarrow 0,
\end{align}
as $n \to \infty$. Therefore, the distribution of ${\bX }$ has HDA as in Definition \ref{defn:hdoa} on $\E^0 = (-\infty, \infty]^2$ with limit measure $\nu^0$ such that for ${\bx } \in \E^0$,
$$ \nu^0\left((x^1, \infty] \times (x^2, \infty] \right) = \left \{ \begin{array}{lr} \frac{1}{2}( (-x^1) + (-x^2) ), & \text{if }  x^1 + x^2 \le 0\\
0 & \text{otherwise.} \end{array} \right.$$
 From \eqref{ex2eqn:hdoa} and \eqref{ex2eqn:hdoa2} it also follows
 that $X^{(2)}$ belongs to the domain of attraction of the reversed
 Weibull distribution \cite[page 59]{resnickbook:2008}. So, the
 distribution of ${\bX }$ has HDA, but does not have HRV and
 furthermore, $\E^0$ is not a subset of $\E$.
}}
\end{ex}

\subsection{Estimation}\label{sec:stanestimation}
Because the marginal distributions are assumed to be the same, the standard case is somewhat unrealistic for applications but it is
important to understand estimation for this case before moving on to 
more realistic scenarios.

To estimate joint tail probabilities, we  first estimate the
limit measure $\nu^0$ given in \eqref{hiddendoasmallcone}. Let, $\{
{\bX }, {\bX }_i, i = 1, 2, \cdots, n \}$ be iid {with a common
  distribution satisfying  \eqref{hiddendoasmallcone}.}
From
\eqref{hiddendoasmallcone}, 
\begin{align}\label{eqn:nu0estimation1}
\frac{1}{k} \sum_{i=1}^n \epsilon_{\left( \frac{\bX_i -
      d(n/k)\bone}{c(n/k)}\right)}(\cdot)
\Rightarrow \nu^0(\cdot)  \qquad (k \to \infty, n/k \to
\infty),
\end{align}
{i}n $M_+(\E^0)$ \citep[page 139]{resnickbook:2007}. From
\eqref{eqn:mindoa},
the distribution of $X^{(2)}$ {determines} $c_n$ and $d_n$. 
The iid data $\{ X^{(2)}_i: i = 1, 2, \cdots, n \}$ {allow}
estimates  (\citep{dehaan:ferreira:2006}, \citep[page 93]{resnickbook:2007}) 
of $c(n/k)$ and $d(n/k)$, denoted by $\hat c(n/k)$ and $\hat
d(n/k)$, satisfying
\begin{align}\label{eqn:an0bn0 convergence}
\frac{c(n/k)}{\hat c(n/k)} \stackrel{P}{\rightarrow} 1, \hskip 1 cm \frac{d(n/k) - \hat d(n/k)}{c(n/k)} \stackrel{P}{\rightarrow} 0;
\end{align}
Therefore, we get the joint convergence
\begin{align}\label{eqn:nu0estimation2}
\left( \frac{1}{k} \sum_{i=1}^n \epsilon_{\left( \frac{\bX_i -
        d(n/k)\bone}{c(n/k)}\right)},
  \frac{d(n/k) - \hat d(n/k)}{c(n/k)}, \frac{c(n/k)}{\hat c(n/k)}
\right) \Rightarrow \left( \nu^0(\cdot), 0, 1 \right)  
\end{align}
{i}n $M_+(\E^0) \times \R^2$. {Apply} the almost surely
continuous map $(\nu(\cdot), b, a) \mapsto \nu(a [ (\cdot) +
b{\bone}] )$ in \eqref{eqn:nu0estimation2}  {and} we get the following proposition:

\begin{prop}\label{prop:nonparametric estimation} 
{{ Let, $\{ {\bX }, {\bX }_i, i \geq 1\}$ be iid with common distribution satisfying
    \eqref{hiddendoasmallcone}. Then, 
\begin{align}\label{eqn:n0estimate}
\widehat{ \nu_n^0 }(\cdot):= \frac{1}{k} \sum_{i=1}^n \epsilon_{\left( \frac{\bX_i - \hat
      d(n/k)\bone}{\hat c(n/k)}\right)}(\cdot) \Rightarrow \nu^0(\cdot) \qquad (k \to
\infty, n/k \to     \infty),
\end{align}
{i}n $M_+(\E^0)$.
}}
\end{prop}

{Estimation of $\nu^0 (\cdot)$  in Proposition \ref{prop:nonparametric
  estimation} 
 does not exploit the semi-parametric structure
discussed in Section \ref{sec:semiparaoflimit} and has the disadvantages
 that (a) there
is no guarantee the estimator $\widehat{ \nu_n^0 }(\cdot)$ is 
 even a member of the class of possible limit measures; and (b) we are
 required to estimate $c(\cdot)$ and $d(\cdot)$.
These disadvantages are overcome using the
semi-parametric structure as was done for
 HRV in
\citep{mitra:resnick:2010a}.
We  need to estimate the the extreme
value index $\gamma^0$  of the distribution of $X^{(2)}$ as well as  the standardized
hidden spectral measure $S^0$.  Since $\{ X^{(2)}_i: i = 1, 2, \cdots,
n \}$ is iid data, estimating $\gamma^0$ of $X^{(2)}$ is a standard procedure \citep[page
65]{dehaan:ferreira:2006} so we concentrate on estimating
$S^0(\cdot)$. A modification of a
ranks method \citep{huang:1992,dehaan:ferreira:2006,
  heffernan:resnick:2005,
resnickbook:2007}
to obtain an estimator of  
$\tilde \nu^0(\cdot)$ avoids the need to estimate $c(\cdot)$ and $d(\cdot)$.
For $i =1, 2, \cdots,n$, define
\begin{align}\label{defn:new_antiranks}
R^{1, (2)}_i = \big{|} \big\{ j: X^{(2)}_j \ge X^1_i \bigr\} \big{|}
\hskip 0.5 cm \hbox{and} \hskip 0.5 cm  R^{2, (2)}_i = \big{|} \bigl\{
  j: X^{(2)}_j \ge X^2_i \bigr\} \big{|}, 
\end{align}
where $|\cdot |$ denotes size of  a set. Note that $0 \le R^{j, (2)}_i
\le n$ for $i = 1, 2, \cdots, n, j = 1, 2$. Also notice that since
$R^{1, (2)}_i \vee R^{2,(2)}_i = {|} \{ j: X^{(2)}_j \ge
  X^{(2)}_i \} {|}$, $1 \le R^{1, (2)}_i \vee R^{2,(2)}_i
\le n$ for $i = 1, 2, \cdots, n$. Proposition \ref{prop:tildenu0}
gives an
estimator of $\tilde \nu^0$ which we can modify to get an estimator of $S^0(\cdot)$.

\begin{prop}\label{prop:tildenu0} 
{We have in $M_+((0, \infty]^2)$,}
\begin{align}\label{eqn:tildenu0estimate}
\widehat{\tilde \nu^0_n} (\cdot) :=\frac{1}{k} \sum_{i=1}^n
\epsilon_{\left(k/R^{1, (2)}_i, \hskip 0.1 cm 
    k/R^{2, (2)}_i \right)}(\cdot) \Rightarrow \tilde \nu^0(\cdot)
\qquad (k \to \infty, n/k \to \infty),
\end{align}
 where \eqref{equihrv} defines 
$\tilde \nu^0(\cdot)$   and \eqref{defn:new_antiranks} defines $R^{1, (2)}_i$ and $R^{2, (2)}_i$.
\end{prop}

\begin{proof}
From \eqref{hiddendoasmallcone} and definition of $\psi^0(\cdot)$
given in \eqref{defn:psi0}, {we have} {i}n $D({(}0, \infty
{]})$,
\begin{align*}
\frac{1}{k} \sum_{i=1}^n \epsilon_{\left((X^{(2)}_i - d(n/k))/c(n/k)
  \right)}((x, \infty]) \Rightarrow {[\psi^0(x)]}^{-1} .
\end{align*}
 Hence \citep[page 58]{resnickbook:2007},  inverse functions also
 converge in distribution {i}n $D_{left}((0, \infty])$, the space
 of left continuous functions with finite right limits,
\begin{align}\label{invconvpsi}
\inf \{ x: \frac{1}{k} \sum_{i=1}^n \epsilon_{\left((X^{(2)}_i -
    d(n/k))/c(n/k) \right)}((x, \infty])  \le 1/s \}\Rightarrow \inf
\{ x: {[\psi^0(x)]}^{-1} \le 1/s \} = {\left(\psi^0
  \right)}^{\leftarrow}(s). 
\end{align}
{Write the order statistics of $\{X_1^{(2)},\dots,X_n^{(2)}\}$ as
$X_{(1)}^{(2)} \geq \dots \geq X_{(n)}^{(2)}$ and}
observe {the left side of \eqref{invconvpsi} is}
\begin{align}\label{eqn:invpsiestcomputation}
\inf \{ x:
\sum_{i=1}^n \epsilon_{\left((X^{(2)}_i - d(n/k))/c(n/k) \right)}((x,
\infty])  \le k/s \} 
= \left(\frac{X^{(2)}_{(\lceil k/s \rceil)} - d(n/k)}{c(n/k)} \right).
\end{align}
From \eqref{eqn:nu0estimation1}, \eqref{invconvpsi} and
\eqref{eqn:invpsiestcomputation} we get that as $k \to \infty$ and
$n/k \to \infty,$ 
\begin{align}\label{eqn:jointnu0conv}
\Bigl( \frac{1}{k} \sum_{i=1}^n \epsilon_{\left( \frac{\bX_i -
        d(n/k)\bone}{c(n/k)}\right)}(\cdot),  &
\frac{ X^{(2)}_{(\lceil k/s \rceil)} ) - d(n/k)}{c(n/k)}, 
\frac{ X^{(2)}_{(\lceil k/t \rceil)} ) - d(n/k)}{c(n/k)} \Bigr)
\nonumber \\
&\Rightarrow 
\left( \nu^0(\cdot), \left( {\left(\psi^0
      \right)}^{\leftarrow}(s), {\left(\psi^0 \right)}^{\leftarrow}(t)
  \right)\right) 
\end{align}
{i}n $M_+(\E^0) \times D_{left}{(}(0,\infty] {)}\times
  D_{left}{(}(0,\infty]{)}$.
{Using the scaling technique as in \citet[page 311]{resnickbook:2007}}
we get from \eqref{eqn:jointnu0conv}  that 
as $k \to \infty$ and $n/k \to \infty,$
\begin{align}\label{eqn:jointnu0conv2}
 \frac{1}{k} \sum_{i=1}^n 1_{\left \{ X^1_i > X^{(2)}_{(\lceil k/{s} \rceil)}, X^2_i > X^{(2)}_{(\lceil k/t \rceil)} \right \} } \Rightarrow \nu^0 \left( \left({\left(\psi^0 \right)}^{\leftarrow}(s), \infty \right] \times \left({\left(\psi^0 \right)}^{\leftarrow}(t), \infty \right] \right) = \tilde \nu^0 \left( (s, \infty] \times (t, \infty] \right),
\end{align}
{i}n $D_{left}{(}0,\infty]{)}\times
  D_{left}{(}(0,\infty] {)}$. Since the left side of
  \eqref{eqn:jointnu0conv2} is
\begin{align*}
\frac{1}{k} \sum_{i=1}^n 1_{\left \{ R^{1, (2)}_i <
     k/s, \hskip 0.1 cm R^{2, (2)}_i < k/t \right \} } 
= \frac{1}{k} \sum_{i=1}^n 1_{\left \{ s <  k/R^{1, (2)}_i, \hskip 0.1 cm t < k/R^{2, (2)}_i \right \} },
 \end{align*} we have proven \eqref{eqn:tildenu0estimate}.
\end{proof}

Proposition  \ref{prop:tildenu0}  yields an estimator of the limit measure $\nu_1 \times
S^0(\cdot)$ and then an estimator of~$S^0(\cdot)$.

\begin{prop}\label{prop:spectralestimation_smallcone}
{{ The convergence in \eqref{eqn:tildenu0estimate} is equivalent to 
\begin{align}\label{eqn:tildenu0decompestimate}
\widehat{\nu_1 \times  S^0}_n(\cdot) :=
\frac{1}{k} \sum_{i=1}^n \epsilon_{\left(\frac{k}{R^{1, (2)}_i \vee
      R^{2, (2)}_i} , \left(\frac{R^{1, (2)}_i \vee R^{2,
          (2)}_i}{R^{1, (2)}_i}, \frac{R^{1, (2)}_i \vee R^{2,
          (2)}_i}{R^{2, (2)}_i} \right) \right)}(\cdot) \Rightarrow
\nu_1 \times  S^0(\cdot) 
\end{align}
in $M_+((0, \infty] \times \delta \aleph^{(2)})$, where $\nu_1 \times  S^0(\cdot)$ is given in Proposition \ref{prop:decomp}.
}}
\end{prop}

\begin{proof}
The proof uses Proposition \ref{prop:decomp} and follows exactly
similar steps as that of Proposition 3.7 of
\cite{mitra:resnick:2010a}. {It is based on the map $\bx \mapsto
(x^{(2)}, \bx/x^{(2)})$.}
\end{proof}

From the convergence in \eqref{eqn:tildenu0decompestimate}, we
construct a consistent estimator of $S^0(\cdot)$:  
\begin{align}\label{eqn:s0pestimate}
\hat{S}_n^0(\cdot) :=
\frac{ \sum_{i=1}^n \epsilon_{\left(\frac{k}{R^{1, (2)}_i \vee R^{2,
          (2)}_i} , \left(\frac{R^{1, (2)}_i \vee R^{2, (2)}_i}{R^{1,
            (2)}_i}, \frac{R^{1, (2)}_i \vee R^{2, (2)}_i}{R^{2,
            (2)}_i} \right) \right)} \left([1, \infty] \times \cdot
  \right)}{\sum_{i=1}^n \epsilon_{\frac{k}{R^{1, (2)}_i \vee R^{2,
        (2)}_i} }([1, \infty])} \Rightarrow S^0(\cdot) 
\end{align}
in $M_+(\delta \aleph^{(2)})$. Hence, we have obtained a consistent estimator for both the extreme value index $\gamma^0$ and the standardized hidden spectral measure $S^0$.

{It is possible that} $R^{j, (2)}_i {=0}$ for some $j = 1, 2$ and
some $i = 1, 2, \cdots, n$ and thus 
division by zero may be indicated in 
\eqref{eqn:s0pestimate}.
 Though theoretically justified, this is not  desirable
when writing code for an estimator. The continuous
bijection $T: \delta \aleph^{(2)} \mapsto [0,1]$ given by $T: {\bx }
\mapsto x^2/(x^1 + x^2)$ provides an instant remedy. We use the
convention that $\infty/\infty = 1$ and $1/\infty = 0$. Using this
transformation, \eqref{eqn:s0pestimate} becomes  
\begin{align}\label{eqn:trs0pestimate}
\frac{ \sum_{i=1}^n \epsilon_{\left(\frac{k}{R^{1, (2)}_i \vee R^{2, (2)}_i} , \left( \frac{R^{1, (2)}_i }{R^{1, (2)}_i + R^{2, (2)}_i} \right) \right)} \left([1, \infty] \times \cdot \right)}{\sum_{i=1}^n \epsilon_{\frac{k}{R^{1, (2)}_i \vee R^{2, (2)}_i} }([1, \infty])} \Rightarrow S^0 \circ T^{-1}(\cdot)
\end{align}
in $M_+([0, 1])$. Since $T$ is a continuous bijection, we  retrieve $S^0$ from 
$S^0 \circ T^{-1}$.

\section{Non-standard hidden domain of attraction}\label{sec:nonstandard}
{To provide more scope for applications, 
the non-standard case no longer assumes that $X^1
\stackrel{d}{=} X^2$. However, we have found that to construct a
coherent estimation theory requires careful consideration of the
definitions.
As in the standard case, the goal is to approximate
 marginal and
joint tail probabilities.}

\subsection{How to proceed?}\label{subsec:how?}
In order for \eqref{eqn:mevt} to hold when {$\bX$ has different
  marginal distributions, one typically needs
different centering and scaling
constants for the two components of $\bX$.  Traditional theory  \cite[page
277, Proposition 5.15]{resnickbook:2008} proceeds by standardizing each
component.
However,  a theory of hidden domain of attraction
that follows this approach encounters  problems in the estimation
procedure that we could not  resolve without strong second order
conditions. 

We deviate from the
traditional MEVT treatment by requiring that both
components in \eqref{doabigcone} have
 the same centering and scaling   but permitting the limit
measure to have one  zero marginal. By a zero marginal, we mean
 that either the limit measure 
$\nu(\cdot)$ in
\eqref{doabigcone} has the property
$$\nu\left( \left \{
    \bz \in \E : z^2 > y \right \} \right) = 0 \qquad ((y,y)
\in \E)$$ 
or the same holds with $z^1$ in place of $z^2$.
This could happen, for instance, if the tail of $X^2$ is lighter than
that of $X^1$ or vice versa.}

 In the non-standard case, if we assume
\eqref{doabigcone}, the limit measure $\nu$ may satisfy:
\begin{enumerate}[(i)]
\item  $\nu$ has a zero second marginal: for $(y,y) \in
  \E$, $\nu\left( \left \{ \bz \in \E : z^2 > y \right \}
  \right) = 0$;
\item  $\nu$ has a zero first marginal: for $(y,y) \in \E$, $\nu\left( \left \{ \bz \in \E : z^1 > y \right \} \right) = 0$,
\item the cases (i) and (ii) do not hold, but $\nu\left( \left \{ \bz \in \E : z^1 > x, z^2 > y \right \} \right) = 0$ for $(x, y) \in \E$,
\item for $(x, y) \in \E$, $\nu\left( \left \{ \bz \in \E : z^1 > x, z^2 > y \right \} \right) > 0$.
\end{enumerate} 

Case (iv) means \eqref{doabigcone} yields
non-zero estimates of the marginal and joint tail probabilities, so
in this case we have no need to define HDA. The definition and analysis of
HDA in case (iii) is the same as the standard case discussed in
Section  
\ref{sec:standard}. {The definition and analysis of HDA are very
  similar for cases (i) and (ii) so focus only on case (i).}

For our definition of HDA,  a relevant state space  is
  $\E^{\sqcap}$ 
where either $\E^{\sqcap} = [-\infty, \infty] \times (-\infty, \infty]$ or
$\E^{\sqcap}= [0, \infty] \times (0, \infty]$. We will see
later that for  case (i) it is possible to find HDA on both the cones $\E^{\sqcap}$
and $\E^0$ in sequence. Thus, compared to the standard case,
our estimation procedure here might involve analyzing HDA on the
additional cone $\E^{\sqcap}$.

\begin{defn}\label{defn:hiddendoamediumcone}
{\rm{The distribution of ${\bX } = (X^1, X^2)$ has hidden domain of
    attraction on the cone $\E^{\sqcap}$ if \eqref{doabigcone} holds
    with the limit measure $\nu$, the second marginal of $\nu$ is a
    zero measure and in addition, there exist 
    constants $ e_n>0$ and $ f_n \in \mathbb{R}$ and a non-zero measure
    $\nu^{\sqcap} \in M_+(\E^{\sqcap})$ such that  
 as $n \to \infty$,
\begin{align}\label{eqn:hiddendoamediumcone}
nP\left[ \left( \frac{{{\bf{X}} -f_n{\bf{1}}}}{e_n} \right) \in
  \cdot \right] \stackrel{v}{\rightarrow} \nu^{\sqcap}(\cdot)\qquad 
\text{in }M_+(\E^{\sqcap}).
\end{align}

}}
\end{defn}
From \eqref{eqn:hiddendoamediumcone} it follows that  for $(y,y) \in \E^{\sqcap}$, as $n \to \infty$,
\begin{equation}\label{eqn:scompdoa}
nP[ X^2 > e_ny + f_n] \to \nu^{\sqcap}\left( \{(u,v)\in E^\sqcap: v>y\} \right).
\end{equation}
Therefore, the distribution of $X^2$, {the second component of
  ${\bX }$,} belongs to the maximal domain of attraction of an extreme
value distribution \cite[page 4]{dehaan:ferreira:2006}. Using
one-dimensional extreme value theory, $\psi^{\sqcap}(y) :=
{\left[\nu^{\sqcap}\left([0, \infty] \times (y, \infty] \right)
  \right]}^{-1}$ must take one of the following forms \cite[page
9]{resnickbook:2008}: 

\begin{align}\label{eqn:formsofpsisqcap}
\begin{array}{lc}
 \psi^{\sqcap}(y) = \left \{ \begin{array}{cc}
y^{1/\gamma^{\sqcap}}, & \hbox{if $y >0$,}\\
0, & \hbox{otherwise,}
\end{array} \right. & \hbox{if $ \gamma^{\sqcap} > 0,$}\\
 \psi^{\sqcap}(y) = e^{y}, \, y \in \mathbb{R},&  \hbox{if $ \gamma^{\sqcap} = 0,$}\\
 \psi^{\sqcap}(y) = \left \{ \begin{array}{cc}
\infty, & \hbox{if $y > 0$,}\\
(-y)^{-1/\gamma^{\sqcap}}, & \hbox{otherwise,}
\end{array} \right. & \hbox{if $ \gamma^{\sqcap} < 0.$}
\end{array}
\end{align}
The parameter $\gamma^{\sqcap}$ in \eqref{eqn:formsofpsisqcap} is the
extreme value index of $X^2$. We {can and always do}  choose $\{ e_n \}$ and $\{ f_n
\}$ in such a way that $\psi^{\sqcap}$ takes one of the above
forms.

\begin{rem}\label{rem:hiddendoamediumcone}
{\rm{ We make a few remarks about Definition \ref{defn:hiddendoamediumcone}.

\begin{enumerate}
\item  {Since \eqref{doabigcone} holds, so does}
  \eqref{eqn:maxdoa} {and therefore} the maximum component
  $X^{(1)}$ belongs to the maximal domain of attraction of some
  extreme value distribution. Since $X^{(1)} \ge X^2$, the convergence
  relation \eqref{doabigcone} {constrains the}
  possible convergences {in}
  \eqref{eqn:hiddendoamediumcone}. For example, if
  {\eqref{eqn:maxdoa} has} 
  $X^{(1)}$ in the Gumbel domain of
  attraction, then the distribution of $X^2$ cannot have a
  regularly varying tail.

\item \label{rem2:hiddendoamediumcone} {T}he distribution of $X^2$ {determines} the cone
  $\E^{\sqcap}$ and \eqref{eqn:scompdoa} yields the scaling and centering constants $\{e_n\}$ and
  $\{f_n \}$. If the distribution of $X^2$ is in the Fr\'{e}chet
  domain of attraction,
  $\E^{\sqcap} = [0, \infty] \times (0, \infty]$ and otherwise, 
  $\E^{\sqcap} = [-\infty, \infty] \times (-\infty, \infty]$. 
\end{enumerate}
}}
\end{rem}

{There are two possibilities for the limit measure}
$\nu^{\sqcap}$ in \eqref{eqn:hiddendoamediumcone}:

\begin{enumerate}[(i)]
\item the limit measure $\nu^{\sqcap}$ puts zero mass on all sets $(x,
  \infty] \times (y, \infty]$ for $(x, y) \in \E^{\sqcap}$; or 
\item the limit measure $\nu^{\sqcap}$ puts non-zero mass on one of the sets  $(x, \infty] \times (y, \infty]$ for $(x, y) \in \E^{\sqcap}$. 
\end{enumerate}

{T}he semi-parametric structure of $\nu^{\sqcap}$ {discussed}
in the next section {implies}
 that for  case (ii), $\nu^{\sqcap}((x, \infty] \times (y,
\infty]) > 0$ for all $(x, y) \in \E^{\sqcap}$. So, in case (ii), we  get non-zero estimates of
joint tail probabilities and since we accomplished our goal there is
no reason to seek further instances of HDA.
However, in case (i), the measure $\nu^\sqcap (\cdot)$ will not
provide non-zero estimates of joint tail probabilities. A potential
solution is that HDA could still exist a smaller cone such as $\E^0$. The
following definition formalizes the concept. {For this definition, the
state space is $\E^0$,
 where $\E^0 = {(-\infty, \infty]}^2$ or $\E^0 = (0, \infty]^2$.}

\begin{defn}\label{defn:nonstansecondstage}
{\rm{ The distribution of ${\bX } = (X^1, X^2)$ has hidden domain of attraction on the cones $\E^{\sqcap}$ and $\E^0$ if Definition \ref{defn:hiddendoamediumcone} holds, the 
the limit measure $\nu^{\sqcap}$ in \eqref{eqn:hiddendoamediumcone} puts zero mass to all sets of the form $(x, \infty] \times (y, \infty]$ for $(x, y) \in \E^{\sqcap}$, and in addition, 
there exist centering and scaling constants $\{ c_n \}$ and $\{d_n \}$ and a non-zero measure $\nu^0 \in M_+(\E^0)$ such that as $n \to \infty$,
\begin{align}\label{eqn:nonstan_hiddendoasmallcone}
nP\left[ \left( \frac{     \bX -d_n \bone       }{  c_n  } \right) \in
  \cdot \, \right] \stackrel{v}{\rightarrow} \nu^0(\cdot) \qquad \text{in }M_+(\E^0).
\end{align}
}}
\end{defn}
As noted before in \eqref{eqn:mindoa}, the scaling and centering
constants $\{ c_n \}$ and $\{ d_n \}$ in
\eqref{eqn:nonstan_hiddendoasmallcone} are characterized by the
distribution of $X^{(2)}$, 
the minimum component of ${\bX }$. Recall the definition of $\psi^0$
given in \eqref{defn:psi0}. As was done in the standard case
discussion, we choose the scaling and centering constants $\{ c_n \}$
and $\{ d_n \}$ in \eqref{eqn:nonstan_hiddendoasmallcone} so that
$\psi^0$ takes one of the forms given in \eqref{eqn:formsofpsi}. Also, 
{whether $\E^0$ in
\eqref{eqn:nonstan_hiddendoasmallcone} 
 is $(-\infty, \infty]^2$ or $(0,
\infty]^2$
 is determined by the
distribution of $X^{(2)}$.}

\subsection{Semi-parametric structure of
  $\nu^{\sqcap}$}\label{sec:nonstansemipara} 
Both  limit measures $\nu^{\sqcap}$ of
\eqref{eqn:hiddendoamediumcone} and $\nu^0$ of
\eqref{eqn:nonstan_hiddendoasmallcone} have semi-parametric
structures. {Since} the semi-parametric structure of $\nu^0$ was
discussed in Section \ref{sec:semiparaoflimit}, {we}
concentrate only on the semi-parametric structure of $\nu^{\sqcap}$ {and
proceed as follows.}

Recall that the distribution{s} of ${\bX }$ and $X^2$ are  $H$ and $H^2$.  Define 
\begin{equation}\label{defn:U2}
U^2(x) = 1/(1 - H^2(x)).
\end{equation}
The following proposition relates \eqref{eqn:hiddendoamediumcone} to a
regular variation condition on $[0, \infty] \times (0, \infty]$.
Its proof is  similar to {that} of Proposition \ref{prop:equihrv} and is omitted.

\begin{prop}\label{prop:nonstanequihrv}
{{Convergence in \eqref{eqn:hiddendoamediumcone} is equivalent to
    {regular variation on the cone $[0, \infty] \times (0, \infty]$},
\begin{equation}\label{nonstanequihrv}
nP\left[ \left( \frac{U^2(X^1)}{n}, \frac{U^2(X^2)}{n} \right) \in
  \cdot \right] \stackrel{v}{\rightarrow} \tilde \nu^{\sqcap}(\cdot)
\qquad (\text{in }M_+([0, \infty] \times (0, \infty]),
\end{equation}
 where \eqref{defn:U2} defines $U^2(\cdot)$  and $\tilde
 \nu^{\sqcap}(\cdot)$ is a Radon measure on $[0, \infty] \times (0,
 \infty]$. The limit measure $\tilde \nu^{\sqcap}(\cdot)$ is related
 to the limit measure in $\nu^{\sqcap}(\cdot)$ in
 \eqref{eqn:hiddendoamediumcone} by the following relation: for $(x^1,
 x^2) \in [0, \infty] \times (0, \infty]$,  
\begin{align}\label{nusqcapandtildenusqcap}
\tilde \nu^{\sqcap}\bigl((x^1, \infty] \times (x^2, \infty]\bigr) &=
\nu^{\sqcap} \Bigl( \bigl({(\psi^{\sqcap} )}^{\leftarrow}(x^1), \infty
\bigr] \times \bigl({(\psi^{\sqcap})}^{\leftarrow}(x^2), \infty \bigr]
\Bigr),\nonumber\\ 
\tilde \nu^{\sqcap} \bigl([0, x^1] \times (x^2, \infty]\bigr) &=
\nu^{\sqcap} \Bigl( \left\{ \bz \in \E^{\sqcap}: z^1 \le
  {\left(\psi^{\sqcap}\right)}^{\leftarrow}(x^1), z^2 >
  {\left(\psi^{\sqcap}\right)}^{\leftarrow}(x^2) \right\} \Bigr). 
\end{align}
The measure $\tilde \nu^{\sqcap}(\cdot)$ satisfies the scaling property:
\begin{align}\label{scaling_nusqcaptilde}
\tilde \nu^{\sqcap}(c \cdot) = c^{-1} \tilde \nu^{\sqcap}( \cdot) \hskip 1 cm c > 0.
\end{align}
}}
\end{prop}

\begin{rem}
{\rm{ 
\begin{enumerate}[(i)]
\item 
On the semi-parametric structure of $\nu^{\sqcap}$:
We will see that a probability measure $S^{\sqcap}$
  on $[0, \infty]$ determines the limit measure $\tilde
  \nu^{\sqcap}$.
The
  parameter $\gamma^{\sqcap}$ and the probability measure $S^{\sqcap}$
  on $[0, \infty]$  determine $\nu^{\sqcap}$, since
given $\gamma^{\sqcap}$ and 
  measure $S^{\sqcap}$, we get 
 the function
  $\psi^{\sqcap}(\cdot)$  in \eqref{eqn:formsofpsisqcap} 
and $\tilde
  \nu^{\sqcap}$ which through
\eqref{nusqcapandtildenusqcap} determines $\nu^{\sqcap}$.

\item {I}f the support of the distribution of $X^2$ is smaller
  than that of $X^1$, then $U^2(X^1)$ could take the value
  $\infty$, {so in this case}
we consider $U^2(X^1)$ as an
  extended random variable.  
\end{enumerate}
}}
\end{rem}

{The method that shows} $\tilde \nu^0([1,
\infty]^2) = 1$ {also  shows} $\tilde \nu^{\sqcap}([0, \infty]
\times [1, \infty]) = 1$. Proposition 4 of
\cite{heffernan:resnick:2007} and Proposition
\ref{prop:nonstanequihrv} {give} {a}  convergence relation
{in new coordinates.}

\begin{prop}\label{prop:nonstandecomp}
{{ The convergence in \eqref{eqn:hiddendoamediumcone} is equivalent to
\begin{equation}\label{eqn:nonstandecomp}
nP\left[ \left( \frac{U^2\left(X^2\right)}{n} ,
    \frac{U^2(X^1)}{U^2(X^2)} \right) \in \cdot \right]
\stackrel{v}{\rightarrow} \nu_1 \times  S^{\sqcap}(\cdot)\qquad
(\text{in }M_+( (0, \infty] \times [0, \infty])),
\end{equation}
where $\nu_1$ is a Pareto measure
on $(0, \infty]$ satisfying $\nu_1((x, \infty]) = x^{-1}$
for $x > 0$, and $S^{\sqcap}$ is a probability measure on $[0,
\infty]$,
called the standardized hidden spectral
measure.  The relation between $\tilde \nu^{\sqcap}$
given in \eqref{equihrv} and $S^{\sqcap}$ is
 \begin{equation}\label{tildenusqcapandssqcap}
 \tilde \nu^{\sqcap} \left( \left \{ {\bx } \in [0, \infty] \times (0, \infty] : x^2 \ge r, x^1/x^2 \in \Lambda \right \} \right) = r^{-1} S^{\sqcap}(\Lambda),
 \end{equation}
which holds for all $r >0$ and all Borel sets $\Lambda \subset [0, \infty]$.
}}
\end{prop}

\subsection{Examples} 
{We give} examples of distributions of $\bX=(X^1,X^2)$,
{$X^1\stackrel{d}{\neq}X^2$}, and  which {have}
HDA.
 In {E}xample \ref{eg:one},
the limit measure $\nu^{\sqcap}$ of \eqref{eqn:hiddendoamediumcone}
puts zero mass {on} all sets of the form $(x^1, \infty] \times (x^2,
\infty]$ for ${\bx } \in \E^{\sqcap}$ and  HDA also holds
on $\E^0$. In {E}xample {\ref{eg:two}},  $\nu^{\sqcap}$ of
\eqref{eqn:hiddendoamediumcone} puts non-zero mass {on} sets of the
form $(x^1, \infty] \times (x^2, \infty]$ for ${\bx } \in
\E^{\sqcap}$. 

\begin{ex}\label{eg:one}
{\rm
 Let $X^1 \sim \exp(1)$, $X^2 \sim \exp(2)$ and $X^1$ and $X^2$ be independent. Then we get
\begin{align*}
n\left(1 - P\left[ X^1 -\log n \le  x^1, X^2 - \log n \le x^2 \right] \right) &= n\left[1 - \left(1 -e^{ -(\log n + x^1)}\right)\left(1 - e^{ -2(\log n + x^2)} \right) \right]\\
& \rightarrow e^{-x^1} \quad \text{as $n \to \infty$},
\end{align*}
which implies \eqref{doabigcone} holds on $\E = [-\infty, \infty]^2
\setminus \{(-\infty, \infty)\}$ with $\nu\left( {\left([-\infty, x^1]
      \times [-\infty, x^2]\right)}^c\right) = e^{-x^1}$
{and $\nu$ puts mass only on $(-\infty,\infty]\times \{-\infty\}$
and} $\nu$ has {zero} second  marginal. {S}o we {seek} HDA on
$\E^{\sqcap}$. {A}s $n \to \infty$, 
\begin{align*}
nP\left[ X^1 -\frac{\log n}{2} \le  x^1, X^2 - \frac{\log n}{2} > x^2 \right] &= n \left(1 -e^{ -(\frac{\log n}{2} + x^1)}\right) e^{ -2(\frac{\log n}{2} + x^2)} \rightarrow e^{-2x^2},
\intertext{and also as $n \to \infty$,}
nP\left[ X^1 -\frac{\log n}{2} >  x^1, X^2 - \frac{\log n}{2} > x^2 \right] &= n e^{ -(\frac{\log n}{2} + x^1)} e^{ -2(\frac{\log n}{2} + x^2)}  \rightarrow 0.
\end{align*}
Thus, HDA exists on $\E^{\sqcap} = [-\infty, \infty] \times (-\infty,
\infty]$ with limit measure $\nu^{\sqcap}$, where
$\nu^{\sqcap}([-\infty, x^1] \times(x^2, \infty]) = e^{-2x^2}$ and
$\nu^{\sqcap}((x^1, \infty] \times(x^2, \infty]) = 0$ for ${\bx } \in
\E^{\sqcap}$ {so $\nu^\sqcap $ concentrates on 
$\{-\infty\}\times   (-\infty,\infty].$} {After peeling away both lines through $-\binfty $,
 we  look} for HDA on $\E^0$. 
{A hint for how to proceed is provided by $X^1\wedge X^2 \sim
  \text{exp}(3)$.}   Note that
  as $n \to \infty$, 
\begin{align*}
nP\left[ X^1 -\frac{\log n}{3} >  x^1, X^2 - \frac{\log n}{3} > x^2 \right] &= n e^{ -(\frac{\log n}{3} + x^1)} e^{ -2(\frac{\log n}{3} + x^2)}  \rightarrow e^{-(x^1 + 2x^2)} .
\end{align*}
Thus HDA exists on $\E^0 =  (-\infty, \infty]^2$ with limit measure $\nu^0$, where $\nu^0( (x^1, \infty] \times (x^2, \infty]) = e^{-(x^1+2x^2)}$ for ${\bx } \in \E^0$.

{F}or this example, {Definition \ref{defn:nonstansecondstage} holds}
and HDA holds on both the cones $\E^{\sqcap}$ and $\E^0$,
}
\end{ex}

\begin{ex}\label{eg:two}
{\rm{ 
Suppose $E_1,E_2,E_3$ are iid exp(1) random variables independent of
$B \sim$ Bernoulli($1/2$) and define
${\bX }$ as
$$ {\bX } = B(E_1, E_3/3) + (1-B)(E_2/2, E_2/2). $$
As $n \to \infty$,
\begin{align*}
n& \left(1 - P\left[ X^1 -\log n \le  x^1, X^2 - \log n \le x^2 \right] \right) \\
&= n\left[1 - \left[\frac{1}{2}\left(1 -e^{ -(\log n + x^1)}\right)\left(1 - e^{ -3(\log n + x^2)} \right) + \frac1{2}\left(1 -e^{ -2(\log n + x^1\wedge x^2)}\right) \right] \right] \rightarrow e^{-x^1},
\end{align*}
which implies \eqref{doabigcone} holds on $\E = [-\infty, \infty]^2
\setminus \{(-\infty, \infty)\}$ with $\nu\left( {\left([-\infty, x^1]
      \times [-\infty, x^2]\right)}^c\right) = e^{-x^1}$ {and
  $\nu$ concentrates on $(-\infty, \infty]\times \{-\infty\}$}. {Thus}
 $\nu$ has {zero} second  marginal and we seek HDA on
$\E^{\sqcap}$. Note that 
\begin{align*}
n&P\left[ X^1 -\frac{\log n}{2} \le  x^1, X^2 - \frac{\log n}{2} > x^2 \right] \\
&= n \left[ \frac1{2}\left(1 -e^{ -(\frac{\log n}{2} + x^1)}\right) e^{ -3(\frac{\log n}{2} + x^2)} + \frac1{2}\left(e^{ -2(\frac{\log n}{2} + x^2)} -e^{ -2(\frac{\log n}{2} + x^1)}\right)1_{\{x^2 < x^1\}} \right] \\
&\rightarrow \frac1{2}\left(e^{-2x^2} - e^{-2x^1}\right)1_{\{x^2 < x^1\}} \quad \text{as } n \to \infty,
\intertext{and also}
n&P\left[ X^1 -\frac{\log n}{2} >  x^1, X^2 - \frac{\log n}{2} > x^2 \right] \\
&= n \left[ \frac1{2}e^{ -(\frac{\log n}{2} + x^1)} e^{ -3(\frac{\log n}{2} + x^2)} + \frac1{2}e^{ -2(\frac{\log n}{2} + x^1 \vee x^2)} \right] \rightarrow \frac1{2}e^{-2(x^1 \vee x^2)} \quad \text{as } n \to \infty.
\end{align*}
Thus, HDA exists on $\E^{\sqcap} = [-\infty, \infty] \times (-\infty,
\infty]$ with limit measure $\nu^{\sqcap}$, where
$\nu^{\sqcap}([-\infty, x^1] \times(x^2, \infty]) =
\frac1{2}\left(e^{-2x^2} - e^{-2x^1}\right)1_{\{x^2 < x^1\}}$ and
$\nu^{\sqcap}((x^1, \infty] \times(x^2, \infty]) = \frac1{2}e^{-2(x^1
  \vee x^2)}$ for ${\bx } \in \E^{\sqcap}$. {In fact, $\nu^\sqcap$
  concentrates on the line $\{(x,x): x\in (-\infty,\infty) {\}}$.}

Since $\nu^{\sqcap}((x^1,
\infty] \times(x^2, \infty]) > 0$ for ${\bx } \in \E^{\sqcap}$, we do
not seek HDA on $\E^0$.}
}
\end{ex}

\subsection{Estimation methods}\label{sec:nonstanestimation}
To estimate joint tail probabilities, we
{require an estimate of}
the limit measure $\nu^{\sqcap}$  given in Definition
\ref{defn:hiddendoamediumcone} {and possibly} $\nu^0$ given in Definition
\ref{defn:nonstansecondstage}. {E}stimation of $\nu^0$ follows the
same steps {as}  in Section \ref{sec:stanestimation} {s}o we
concentrate {on}  {estimating} $\nu^{\sqcap}$. Let, $\{ {\bX },
{\bX }_i, i = 1, 2, \cdots, n \}$ be iid where the distribution of
${\bX }$ satisfies \eqref{eqn:hiddendoamediumcone}. From
\eqref{eqn:hiddendoamediumcone} we get \citep[page
139]{resnickbook:2007}  {i}n $M_+(\E^{\sqcap})$
\begin{align}\label{eqn:nusqcapestimation1}
\frac{1}{k} \sum_{i=1}^n \epsilon_{\left( \frac{X^1_i -
      f(n/k)}{e(n/k)}, \frac{X^2_i - f(n/k)}{e(n/k)}\right)}(\cdot)
\Rightarrow \nu^{\sqcap}(\cdot)  \qquad (k \to \infty, \,n/k
\to \infty).
\end{align}
{We know} from \eqref{eqn:scompdoa} that the distribution of
 $X^2$ characterizes $\{ e_n \}$ and $\{ f_n \}$ and from the iid data
 $\{ X^2_i: i = 1, 2, \cdots, n \},$ we {can construct} estimat{ors} of
 $e(n/k)$ and $f(n/k)$ denoted by $\hat e(n/k)$ and $\hat f(n/k)$  \citep[page 93]{resnickbook:2007}
 such that  
\begin{align}\label{eqn:enfn convergence}
\frac{e(n/k)}{\hat e(n/k)} \stackrel{P}{\rightarrow} 1, \hskip 1 cm \frac{f(n/k) - \hat f(n/k)}{e(n/k)} \stackrel{P}{\rightarrow} 0.
\end{align}
{Since the limits in \eqref{eqn:enfn convergence} are
  constants,}
we get  joint convergence in $M_+(\E^{\sqcap}) \times \R^2$,
\begin{align}\label{eqn:nusqcapestimation2}
\left( \frac{1}{k} \sum_{i=1}^n \epsilon_{\left( \frac{X^1_i -
        f(n/k)}{e(n/k)}, \frac{X^2_i - f(n/k)}{e(n/k)}\right)},
  \frac{f(n/k) - \hat f(n/k)}{e(n/k)}, \frac{e(n/k)}{\hat e(n/k)}
\right) \Rightarrow \left( \nu^{\sqcap}(\cdot), 0, 1 \right)  
\end{align}
{Apply the continuous mapping theorem to {\eqref{eqn:nusqcapestimation2}} using the map}
 $(\nu(\cdot), b, a) \mapsto \nu(a [ (\cdot) + b] )$ to get in $M_+(\E^{\sqcap})$ 
\begin{align}\label{eqn:nsqcapestimate}  
{\widehat{\nu_n^\sqcap}(\cdot):=}
  \frac{1}{k} \sum_{i=1}^n \epsilon_{\left( \frac{X^1_i - \hat
      f(n/k)}{\hat e(n/k)}, \frac{X^2_i - \hat f(n/k)}{\hat
      e(n/k)}\right)}(\cdot) \Rightarrow \nu^{\sqcap}(\cdot)\qquad
(k \to \infty,\, n/k \to
\infty).
\end{align}


This {estimator of $\nu^\sqcap$} is non-parametric and {as} in
Section \ref{sec:stanestimation}, we  exploit the
semi-parametric structure of $\nu^{\sqcap}$ by {estimating}
$\gamma^{\sqcap}$ and the standardized hidden spectral measure
$S^{\sqcap}$. The parameter $\gamma^{\sqcap}$ is the extreme value
index of the distribution of $X^2$ so {e}stimating {this} from  iid data $\{ X^2_i: i = 1, 2,
\cdots, n \}$ is standard \citep[page 65]{dehaan:ferreira:2006}. 
{We obtain an estimator of} 
 $S^{\sqcap}(\cdot)$ by
{modifying \eqref{defn:new_antiranks} to account for the difference
  between $\E^0$ and $\E^\sqcap$}.
{Since the first step is to}
construct a consistent estimator of $\tilde \nu^{\sqcap}(\cdot)$
{defined in \eqref{nonstanequihrv},}
{define}
\begin{align}\label{defn:nonstannew_antiranks}
R^{1, 2}_i : = \Big{|} \left\{ j: X^2_j \ge X^1_i \right\} \Big{|}
\hskip 0.5 cm \hbox{and} \hskip 0.5 cm  R^{2, 2}_i := \Big{|} \left\{
  j: X^2_j \ge X^2_i \right\} \Big{|}, \qquad (i = 1, 2, \cdots,n,)
\end{align}
where $|\cdot |$ denotes size of  a set. {Observe} $R^{2, 2}_i$ is
just the anti-rank of $X^2_i$ and thus
 $1 \le R^{2, 2}_i \le n$ for $i = 1, 2, \cdots, n$.  Also,  $0 \le
 R^{1, 2}_i \le n$ for $i = 1, 2, \cdots, n$. 
{A}n estimator of $\tilde \nu^{\sqcap}$ is obtained from the
convergence in
$M_+([0, \infty] \times (0, \infty])$:
\begin{align}\label{eqn:tildenusqcapestimate}
\widehat{\tilde
\nu_n^{\sqcap}}(\cdot):=
\frac{1}{k} \sum_{i=1}^n \epsilon_{\left(k/R^{1, 2}_i, \hskip 0.1 cm
    k/R^{2, 2}_i \right)}(\cdot) \Rightarrow \tilde
\nu^{\sqcap}(\cdot) \qquad (k \to \infty, n/k \to \infty).
\end{align}
{The verification of \eqref{eqn:tildenusqcapestimate} follows the steps
used in the proof of Proposition \ref{prop:tildenu0}.}
{Changing coordinate system in \eqref{eqn:tildenusqcapestimate} leads
to an estimator of }
 $\nu_1 \times  S^{\sqcap}(\cdot)$ from the convergence
{in $M_+((0, \infty] \times [0, \infty])$}
\begin{align}\label{eqn:tildenusqcapdecompestimate}
\widehat{\nu_1 \times  S^{\sqcap}}_n(\cdot):=
\frac{1}{k} \sum_{i=1}^n \epsilon_{\left({k}/{ R^{2, 2}_i} ,
   {R^{2, 2}_i}/{R^{1, 2}_i} \right)}(\cdot) \Rightarrow \nu_1
\times  S^{\sqcap}(\cdot) \qquad (k \to \infty, n/k \to \infty)
\end{align}
and this produces an estimator of $S^\sqcap$ since {i}n $M_+([0, \infty])$,
\begin{align}\label{eqn:ssqcappestimate}
\widehat{S_n^{\sqcap}}(\cdot) :=
\frac{ \sum_{i=1}^n \epsilon_{\left({k}/{
R^{2, 2}_i} , {R^{2, 2}_i}/{R^{1, 2}_i} \right)} \left([1, \infty]
\times \cdot \right)}{\sum_{i=1}^n \epsilon_{{k}/{R^{2, 2}_i}
}([1, \infty])} \Rightarrow S^{\sqcap}(\cdot) \qquad (k \to \infty, n/k \to \infty).
\end{align}
{This estimator may be modified as in \eqref{eqn:trs0pestimate}}
using the continuous bijection $TR:[0,\infty]\mapsto [0,1]$ defined by
$TR: x\mapsto x/(1+x)$ to get in $M_+([0,1])$,
\begin{align}\label{eqn:trspestimate}
\widehat{S_n^{\sqcap}}(\cdot) \circ TR^{-1}
\Rightarrow S^\sqcap \circ TR^{-1}(\cdot) .
\end{align}

This summarizes how to obtain consistent estimators for extreme value
index $\gamma^{\sqcap}$ and the standardized hidden spectral measure
$S^{\sqcap}$. 

\section{Detection of HDA}\label{sec:detectionestimation}
Since HDA is a generalization of HRV,  it is not surprising that the
detection  techniques have similarities to
{those used for} HRV; see \citet{mitra:resnick:2010a} and 
\citet[pages 316-340]{resnickbook:2007}. However,  we  deviate from the standard
MEVT by assuming \eqref{doabigcone} instead of \eqref{eqn:mevt} and so
proceed carefully.
A first step is to detect the presence of asymptotic independence,
which traditionally has been done with
a density plot of a spectral measure after non-parametric
transformation to Pareto scale \cite[pages
316-321]{resnickbook:2007}.  When asymptotic independence is present, we
seek HDA.  We consider how to define an appropriate spectral measure
for detection of HDA.

Define $U^{(1)}(\cdot) = 1/(1 - H^{(1)}(\cdot)$, where $H^{(1)}$ is
the distribution function of $X^{(1)}$. From \eqref{doabigcone} {we get
on $(0, \infty]^2$ that}
\begin{equation}\label{eqn:equihrvbigcone}
nP\left[ \left( {U^{(1)}(X^1)}/{n}, {U^{(1)}(X^2)}/{n} \right) \in \cdot \right] \stackrel{v}{\rightarrow} \tilde \nu(\cdot),
\end{equation}
{where $\tilde \nu(\cdot)$ is a Radon measure on $(0, \infty]^2$
 related to the limit measure  $\nu(\cdot)$ in \eqref{doabigcone} by }
\begin{equation}\label{nuandtildenu}
\tilde \nu((x^1, \infty] \times (x^2, \infty] ) = \nu \left(
  \left(\psi^{\leftarrow}(x^1), \infty \right] \times
  \left(\psi^{\leftarrow}(x^2), \infty \right] \right)\qquad (\bx  \in (0, \infty]^2),
\end{equation}
{and} $\psi$ is defined in \eqref{defn:psi} as $\psi(y) := \nu(
\{{\bf{z}} \in \E: z^{(1)} > y \}).$
 The measure $\tilde \nu(\cdot)$ satisfies the scaling
\begin{align}\label{scaling_nutilde}
\tilde \nu( c \cdot) = c^{-1} \tilde \nu( \cdot), \hskip 1 cm c > 0,
\end{align}
and convergence in \eqref{eqn:equihrvbigcone} is equivalent to 
\begin{equation}\label{eqn:decompbigcone}
nP\left[ \left( \frac{U^{(1)}\left(X^{(1)}\right)}{n} , \left( \frac{U^{(1)}(X^1)}{U^{(1)}(X^{(1)})},  \frac{U^{(1)}(X^2)}{U^{(1)}(X^{(1)})} \right) \right) \in \cdot \right] \stackrel{v}{\rightarrow} \nu_1 \times  S(\cdot),
\end{equation}
on $(0, \infty] \times \delta \aleph^{(1)}$, where $\nu_1((y, \infty]) = y^{-1}$ for
$y > 0$, $\delta \aleph^{(1)} =  \{ {\bx }  \in (0, \infty]^2 :
x^{(1)} = 1\}$ and $S$ is a probability measure on $\delta
\aleph^{(1)}$. {The} standardized spectral measure $S$
{is related to}  $\tilde \nu$ in \eqref{eqn:equihrvbigcone}
by
 \begin{equation}\label{tildenuands}
 \tilde \nu \left( \left \{ {\bx } \in (0, \infty]^2 : x^{(1)} \ge r,
     {\bx }/x^{(1)} \in \Lambda \right \} \right) = r^{-1} S(\Lambda),
 \qquad r>0,\, \text{Borel set }\Lambda \subset \delta \aleph^{(1)}.
 \end{equation}
{To
estimate this} measure $S(\cdot)$, we define variants of the anti-ranks
\begin{align}\label{defn:nonstanbig_antiranks}
R^{1, (1)}_i = \bigl{|} \bigl\{ j: X^{(1)}_j \ge X^1_i \bigr\} \bigr{|}
\hskip 0.5 cm \hbox{and} \hskip 0.5 cm  R^{2, (1)}_i = \bigl{|} \bigl\{
  j: X^{(1)}_j \ge X^2_i \bigr\} \bigr{|} \qquad (1\leq i \leq n).
\end{align}
A consistent estimator of $S(\cdot)$ is obtained from the convergence in $M_+(\delta \aleph^{(1)})$
\begin{align}\label{eqn:sestimate}
\hat{S_n} :=
\frac{ \sum_{i=1}^n \epsilon_{\left(\frac{k}{R^{1, (1)}_i \wedge R^{2,
          (1)}_i} , \left(\frac{R^{1, (1)}_i \wedge R^{2,
            (1)}_i}{R^{1, (1)}_i}, \frac{R^{1, (1)}_i \wedge R^{2,
            (1)}_i}{R^{2, (1)}_i} \right) \right)} \left([1, \infty]
    \times \cdot \right)}{\sum_{i=1}^n \epsilon_{\frac{k}{R^{1, (1)}_i
      \wedge R^{2, (1)}_i} }([1, \infty])} \Rightarrow S(\cdot),
\end{align}
$(k \to \infty,\, n/k \to \infty)$. The continuous bijection $T: \delta \aleph^{(1)} \mapsto [0,1]$ given by $T: {\bx } \mapsto x^2/(x^1 + x^2)$ transforms \eqref{eqn:sestimate} to 
$\hat{S_n} \circ T^{-1}  \Rightarrow S \circ T^{-1}(\cdot) $ 
in $M_+([0, 1])$. A density  plot of $\hat{S_n} \circ T^{-1}$  {is}
easier to analyze because $[0,1]$ is a nicer space than
$\delta \aleph^{(1)}$. 

Analyzing the density plot using the points of  $\hat{S_n} \circ T^{-1} $
should yield evidence falling into the following categories:
\begin{enumerate}[(i)]
\item The distribution $S \circ T^{-1}$ concentrates near  $0$, so
  remove {$\{(x,y) \in \R^2: y = -\infty\}$} and seek HDA on
  {$\E^{\sqcap} = [-\infty, \infty] \times (-\infty, \infty]$ or its
    first quadrant analogue, depending on the distribution of the
    second component of the random vector; 
    see Remark \ref{rem:hiddendoamediumcone}
    \eqref{rem2:hiddendoamediumcone}.} 
\item The distribution $S \circ T^{-1}$ concentrates near   $1$, so
 remove {$\{(x,y) \in \R^2: x = -\infty\}$} and seek HDA on
 {$(-\infty, \infty] \times [-\infty, \infty]$ or its first quadrant
   analogue depending on the distribution of the first component of the random
   vector.} 
  
\item The distribution $S \circ T^{-1}$  concentrates near  $0$ and
  $1$, so  remove {$\{(x,y) \in \R^2: x = -\infty\} \cup \{(x,y) \in
    \R^2: y = -\infty\}$} and seek HDA on {$\E^0 = (-\infty, \infty]
    \times (-\infty, \infty]$ or its first quadrant analogue,
    depending on 
the distribution of the smallest
    component of the random vector; see Remark \ref{rem:hdoa}
    \eqref{rem3:hdoa}.} 
  
  \item The distribution $S \circ T^{-1}$ does not have any of the above
  properties; {we have no evidence for asymptotic independence} and we do not consider HDA.
\end{enumerate}

{We summarize our detection strategy and for concreteness assume
$S \circ T^{-1}$ satisfies category (i):
Check whether $X^2$ belongs to some maximal domain of attraction using
a Hill or Pickands plots. If so, conclude  HDA exists on
$\E^{\sqcap}$.  Then consider
whether  HDA exists also on $\E^0$
by examining a similar kernel density plot formed by using the points
of the estimator of $S^{\sqcap} \circ {(TR)}^{-1}$ given in
\eqref{eqn:trspestimate}. 
} 

\section{Conclusion}\label{sec:conclusion}
We defined HDA {as} a generalization of HRV {and} have shown by
example that for some random vectors, HDA exists but HRV does
not. Using similar methods as in HRV, we outlined detection and
estimation methods for HDA. These methods are given to show what is
possible and to emphasize there is a gap that such methods can
  fill to provide improved estimates of probability of simultaneous
  exceedance by components of a risk vector. 
However, we have not implemented the methods nor demonstrated
utility by analyzing data. This will come in the
  future.

Our discussion here is
{restricted to  two-dimensions. As}
observed for HRV \citep{mitra:resnick:2010a},  extensions
to higher dimensions are not always straightforward and involve
subtleties.
In particular, in higher dimensions there are many more ways domains
of attraction could be hidden and many more subspaces to explore for
behavior that helps to estimate risk probabilities.

As with HRV \citep{mitra:resnick:2010a}, our detection and estimation
methods are {exploratory} and  our estimators are only provably
consistent. More formal statistical theory is needed to turn
exploratory methods into confirmatory ones.


\bibliography{/Users/sidresnick/Documents/SidFiles/bibfile}

\def\cprime{$'$}
\begin{thebibliography}{17}
\providecommand{\natexlab}[1]{#1}
\providecommand{\url}[1]{\texttt{#1}}
\expandafter\ifx\csname urlstyle\endcsname\relax
  \providecommand{\doi}[1]{doi: #1}\else
  \providecommand{\doi}{doi: \begingroup \urlstyle{rm}\Url}\fi

\bibitem[Balkema and de~Haan(1972)]{balkema:dehaan:1972}
A.~A. Balkema and L.~de~Haan.
\newblock On {R}. von {M}ises' condition for the domain of attraction of
  {$\exp(-e^{-x})^{1}$}.
\newblock \emph{Ann. Math. Statist.}, 43:\penalty0 1352--1354, 1972.
\newblock ISSN 0003-4851.

\bibitem[Billingsley(1999)]{billingsley:1999}
P.~Billingsley.
\newblock \emph{Convergence of Probability Measures}.
\newblock John Wiley \& Sons Inc., New York, second edition, 1999.
\newblock ISBN 0-471-19745-9.
\newblock A Wiley-Interscience Publication.

\bibitem[Bruun and Tawn(1998)]{bruun:tawn:1998}
J.T. Bruun and J.A. Tawn.
\newblock {Comparison of approaches for estimating the probability of coastal
  flooding.}
\newblock \emph{J. R. Stat. Soc., Ser. C, Appl. Stat.}, 47\penalty0
  (3):\penalty0 405--423, 1998.

\bibitem[de~Haan(1978)]{dehaan:1978}
L.~de~Haan.
\newblock A characterization of multidimensional extreme-value distributions.
\newblock \emph{Sankhy\=a Ser. A}, 40\penalty0 (1):\penalty0 85--88, 1978.
\newblock ISSN 0581-572X.

\bibitem[de~Haan and de~Ronde(1998)]{dehaan:deronde:1998}
L.~de~Haan and J.~de~Ronde.
\newblock Sea and wind: multivariate extremes at work.
\newblock \emph{Extremes}, 1\penalty0 (1):\penalty0 7--46, 1998.

\bibitem[de~Haan and Ferreira(2006)]{dehaan:ferreira:2006}
L.~de~Haan and A.~Ferreira.
\newblock \emph{Extreme Value Theory: An Introduction}.
\newblock Springer-Verlag, New York, 2006.

\bibitem[Heffernan and Resnick(2005)]{heffernan:resnick:2005}
J.E. Heffernan and S.I. Resnick.
\newblock Hidden regular variation and the rank transform.
\newblock \emph{Adv.\ Appl.\ Prob.}, 37\penalty0 (2):\penalty0 393--414, 2005.

\bibitem[Heffernan and Resnick(2007)]{heffernan:resnick:2007}
J.E. Heffernan and S.I. Resnick.
\newblock Limit laws for random vectors with an extreme component.
\newblock \emph{Ann. Appl. Probab.}, 17\penalty0 (2):\penalty0 537--571, 2007.
\newblock ISSN 1050-5164.
\newblock \doi{10.1214/105051606000000835}.

\bibitem[Huang(1992)]{huang:1992}
Xin Huang.
\newblock \emph{Statistics of Bivariate Extreme {V}alues}.
\newblock Ph.{D}. thesis, {T}inbergen {I}nstitute {R}esearch {S}eries 22,
  Erasmus University Rotterdam, Postbus 1735, 3000DR, Rotterdam, The
  Netherlands, 1992.

\bibitem[Ledford and Tawn(1996)]{ledford:tawn:1996}
A.W. Ledford and J.A. Tawn.
\newblock Statistics for near independence in multivariate extreme values.
\newblock \emph{Biometrika}, 83\penalty0 (1):\penalty0 169--187, 1996.
\newblock ISSN 0006-3444.

\bibitem[Ledford and Tawn(1998)]{ledford:tawn:1998}
A.W. Ledford and J.A. Tawn.
\newblock Concomitant tail behaviour for extremes.
\newblock \emph{Adv. in Appl. Probab.}, 30\penalty0 (1):\penalty0 197--215,
  1998.
\newblock ISSN 0001-8678.

\bibitem[Mitra and Resnick(2010)]{mitra:resnick:2010a}
A.~Mitra and S.I. Resnick.
\newblock {Hidden Regular Variation: Detection and Estimation}.
\newblock \emph{Arxiv preprint arXiv:1001.5058}, 2010.

\bibitem[Poon et~al.(2003)Poon, Rockinger, and Tawn]{poon:rockinger:tawn:2003}
S.-H Poon, M.~Rockinger, and J.~Tawn.
\newblock Modelling extreme-value dependence in international stock markets.
\newblock \emph{Statist. Sinica}, 13\penalty0 (4):\penalty0 929--953, 2003.
\newblock ISSN 1017-0405.
\newblock Statistical applications in financial econometrics.

\bibitem[Resnick(2002)]{resnick:2002a}
S.I. Resnick.
\newblock Hidden regular variation, second order regular variation and
  asymptotic independence.
\newblock \emph{Extremes}, 5\penalty0 (4):\penalty0 303--336 (2003), 2002.
\newblock ISSN 1386-1999.

\bibitem[Resnick(2007)]{resnickbook:2007}
S.I. Resnick.
\newblock \emph{Heavy Tail Phenomena: Probabilistic and Statistical Modeling}.
\newblock Springer Series in Operations Research and Financial Engineering.
  Springer-Verlag, New York, 2007.
\newblock ISBN: 0-387-24272-4.

\bibitem[Resnick(2008)]{resnickbook:2008}
S.I. Resnick.
\newblock \emph{Extreme Values, Regular Variation and Point Processes}.
\newblock Springer, New York, 2008.
\newblock ISBN 978-0-387-75952-4.
\newblock Reprint of the 1987 original.

\bibitem[Smith(2003)]{smith:2003}
R.L. Smith.
\newblock Statistics of extremes, with applications in environment, insurance
  and finance.
\newblock In B.~Finkenstadt and H.~Rootz\'en, editors, \emph{SemStat: Seminaire
  Europeen de Statistique, Exteme Values in Finance, Telecommunications, and
  the Environment}, pages 1--78. Chapman-Hall, London, 2003.

\end{thebibliography}
  \end{document}